\documentclass[12pt,reqno]{amsart}

\usepackage{amsmath,amssymb,latexsym,amscd,amsthm} 
\usepackage{graphicx}
\usepackage{fullpage}
\usepackage{color}
\usepackage[square, numbers]{natbib}   
\usepackage{fourier}
\usepackage{comment}
\usepackage[latin1]{inputenc}
\usepackage[T1]{fontenc}
 \usepackage{tikz}
 \usetikzlibrary{topaths,calc}
\newtheorem{theorem}{Theorem}[section]  
\newtheorem{corollary}[theorem]{Corollary}
\newtheorem{proposition}[theorem]{Proposition}
\newtheorem{claim}{Claim}
\newtheorem{lemma}[theorem]{Lemma}
\theoremstyle{definition}

\newtheorem*{question}{Question}

\theoremstyle{remark}
\newtheorem*{remark}{Remark}
\def\zero{{\bf 0}}
\reversemarginpar
\title{A combinatorial approach to colourful simplicial depth}

\author{Antoine Deza}
\address{Advanced Optimization Laboratory, Department of Computing and Software, 
McMaster University, Hamilton, Ontario, Canada}
\email{deza@mcmaster.ca}

\author{Fr\'ed\'eric Meunier}
\address{Universit\'e Paris Est, CERMICS, 6-8 avenue Blaise Pascal, Cit\'e Descartes, 77455 Marne-la-Vall\'ee, Cedex 2, France}
\email{frederic.meunier@cermics.enpc.fr}

\author{Pauline Sarrabezolles}
\address{Universit\'e Paris Est, CERMICS, 6-8 avenue Blaise Pascal, Cit\'e Descartes, 77455 Marne-la-Vall\'ee, Cedex 2, France}
\email{pauline.sarrabezolles@cermics.enpc.fr}



\date{\today}

\def\bara{B\'ar\'any}
\def\mato{Matou\v sek}
\def\cara{Carath\'eodory}

\def\R{\mathbb{R}}

\def\S{\mathbf{S}}

\def\conv{\operatorname{conv}}

\def\zero{{\bf 0}}

\begin{document}

\begin{abstract}
The colourful simplicial depth conjecture states that any point in the convex hull of each of $d+1$ sets, or colours, of $d+1$ points in general position in $\R^d$ is contained in at least $d^2+1$ simplices with one vertex from each set. We verify the conjecture in dimension 4 and strengthen the known lower bounds in higher dimensions. These results are obtained using a combinatorial generalization of colourful point configurations called octahedral systems. We present  properties of octahedral systems generalizing earlier results on colourful point configurations and exhibit an octahedral system which cannot arise from a colourful point configuration. The number of octahedral systems is also given.
\end{abstract}

\maketitle

\section{Introduction}
\subsection{Preliminaries}
An $n$-uniform hypergraph is said to be {\em $n$-partite} if its vertex set is the disjoint union of $n$ sets $V_1,\ldots,V_n$ and each edge intersects each $V_i$ at exactly one vertex. Such a hypergraph is an $(n+1)$-tuple $(V_1,\ldots,V_n,E)$ where $E$ is the set of edges.  An {\em octahedral system} $\Omega$ is an $n$-uniform $n$-partite hypergraph $(V_1,\ldots,V_n,E)$ with $|V_i|\geq 2$ for  $i=1,\ldots,n$ and satisfying the following {\em parity condition}: the number of edges of $\Omega$ induced by $X\subseteq\bigcup_{i=1}^nV_i$ is even if $|X\cap V_i|=2$ for $i=1,\ldots,n$.

A {\em colourful point configuration} in $\mathbb{R}^d$ is a collection of $d+1$ sets, or colours, 
$\mathbf{S}_1,\ldots,\mathbf{S}_{d+1}$. A {\em colourful simplex} is defined as the convex hull of a subset $S$ of $\bigcup_{i=1}^{d+1}\mathbf{S}_i$ with $|S\cap\mathbf{S}_i|=1$  for $i=1,\ldots,d+1$. The Octahedron Lemma~\cite{DHST06} states that, given a subset $X\subseteq\bigcup_{i=1}^{d+1}\mathbf{S}_i$ of points such that $|X\cap\mathbf{S}_i|=2$ for $i=1,\ldots,d+1$, there is an even number of colourful simplices generated by $X$ and containing the origin $\zero$. Therefore, the hypergraph $\Omega=(V_1,\ldots,V_{d+1},E)$, with $V_i=\mathbf{S}_i$ for $i=1,\ldots,d+1$ and where the edges in $E$ correspond to the colourful simplices containing $\zero$ forms an octahedral system. This property motivated {\bara} to suggest octahedral systems as a combinatorial generalization of colourful point configurations, see~\cite{DSX11}.

Let $\mu(d)$ denote the minimum number of colourful simplices containing $\zero$ over all colourful point configurations satisfying $\zero\in\bigcap_{i=1}^{d+1}\conv(\mathbf{S}_i)$ and $|\mathbf{S}_i|=d+1$ for $i=1,\ldots,d+1$. B\'ar\'any's colourful Carath\'eodory theorem~\cite{Bar82} states that $\mu(d)\geq 1$. The quantity $\mu(d)$ was investigated in~\cite{DHST06} where it is shown that $2d \le \mu(d) \le d^2+1$, that $\mu(d)$ is even for odd $d$, and that $\mu(2)=5$. This paper also conjectures that $\mu(d)= d^2+1$ for all $d \ge 1$.  Subsequently, {\bara} and {\mato}~\cite{BM06} verified the conjecture for $d=3$ and provided a lower bound of  $\mu(d) \ge \max(3d, \left\lceil \frac{d(d+1)}{5} \right\rceil)$ for $d \ge 3$, while Stephen and Thomas~\cite{ST06} independently proved that $\mu(d) \ge \left\lfloor \frac{(d+2)^2}{4}\right\rfloor$,  before Deza, Stephen, and Xie~\cite{DSX11}  showed that $\mu(d) \ge \left\lceil \frac{(d+1)^2}{2} \right\rceil$. The lower bound was slightly improved in dimension 4 to $\mu(4)\geq 14$ via  a computational approach presented in~\cite{DSX12}.

An octahedral system arising from a colourful point configuration $\S_1,\ldots,\S_{d+1}$, such that $\zero\in\bigcap_{i=1}^{d+1}\conv(\S_i)$ and $|\S_i|=d+1$ for all $i$, is without isolated vertex; that is, each vertex belongs to at least one edge. Indeed, according to a strengthening of the colourful Carath\'eodory theorem \cite{Bar82}, any point of such a colourful configuration is the vertex of at least one colourful simplex containing $\zero$. Theorem~\ref{main}, whose proof is given in  Section~\ref{proofs}, provides a lower bound for the number of edges of an octahedral system without isolated vertex.

\begin{theorem}\label{main}
An octahedral system without isolated vertex and with $|V_1|=|V_2|=\ldots=|V_n|=m$ has at least  $\frac{1}{2}m^2+\frac{5}{2}m-11$ edges for $4\leq m\leq n$ . 
\end{theorem}

Setting $m=n=d+1$ in Theorem~\ref{main} yields a strengthening of the lower bound for $\mu(d)$ given in Corollary~\ref{coro1}.
\begin{corollary}\label{coro1}
$\mu(d)\geq \frac{1}{2}d^2+\frac{7}{2}d-8$ for $d\geq 4$. 
\end{corollary}

Corollary~\ref{coro1} improves the known lower bounds for $\mu(d)$ for all $d \ge 5$. 
Refining the combinatorial approach for small instances in Section~\ref{mu4}, we show that $\mu(4)=17$, i.e. 
the conjectured equality $\mu(d)=d^2+1$ holds in dimension 4, see Proposition~\ref{prop:mu4}.
Properties of octahedral systems generalizing earlier results on colourful point configurations are presented in Section~\ref{OSprop}.
We answer open questions raised in~\cite{CDSX11} in Section~\ref{Q&A} by
determining in Theorem~\ref{thm:number} the number of distinct octahedral systems with given $|V_i|$'s, and by showing that the octahedral system 
given in Figure~\ref{fig:omega9} cannot arise from a colourful point configuration.

{\bara}'s sufficient condition for the existence of a colourful simplex containing $\zero$ has been recently generalized 
in~\cite{AB+09,HPT08,MD12}.
The related algorithmic question of finding a colourful simplex containing $\zero$ is presented and studied 
in~\cite{BO97,DHST08}.
We refer to ~\cite{Gro10,Kar12} for a recent breakthrough for a monocolour version.
\subsection{Definitions}
Let $E[X]$ denote the set of edges induced by a subset $X$ of the vertex set $\bigcup_{i=1}^nV_i$ of an octahedral system $\Omega=(V_1,\ldots,V_n,E)$. The degree of $X$, denoted by $\deg_{\:\Omega}(X)$, is the number of edges containing $X$. An octahedral system $\Omega=(V_1,\ldots,V_n,E)$ with $|V_i|=m_i$ for $i=1,\ldots,n$ is called a $(m_1,\ldots,m_n)$-octahedral system. Given an octahedral system $\Omega=(V_1,\ldots,V_n,E)$, a subset $T\subseteq\bigcup_{j=1}^nV_j$ is a {\em transversal} of $\Omega$ if $|T|=n-1$ and $|T\cap V_j|\leq 1$ for  $j=1,\ldots,n$. The set $T$ is called an {\em $i$-transversal} if $i$ is the unique index such that $|T\cap V_i|=0$. Let $\nu(m_1,\ldots,m_n)$ denote the minimum number of edges over all $(m_1,\ldots,m_n)$-octahedral systems without isolated vertex. The minimum number of edges over all $(d+1,\ldots,d+1)$-octahedral systems has been considered by Deza et al.~\cite{CDSX11} where this quantity is denoted by $\nu(d)$.
By a slight abuse of notation, we identify $\nu(d)$ with $\nu(\underbrace{d+1,\ldots ,d+1}_{d+1 \mbox{ \scriptsize times }})$. We have $\mu(d)\geq\nu(d)$, and the inequality is hypothesized to hold with equality. 

Throughout the paper, given an octahedral system $\Omega=(V_1,\ldots,V_n,E)$, the {\em parity property} refers to the evenness  of $|E[X]|$ if $|X\cap V_i|=2$ for $i=1,\ldots,n$. In a slightly weaker form, the parity property refers to the following observation: If $e$ is an edge, $T$ an $i$-transversal disjoint from $e$, and $x$ a vertex in $V_i \setminus e$, then there is an edge distinct from $e$ in $e\cup T \cup\{x\}$.

Let $D(\Omega)$ be the directed graph $(V,A)$  associated to $\Omega=(V_1,\ldots,V_n,E)$ with vertex set $V:=\bigcup_{i=1}^nV_i$ and where $(u,v)$ is an arc in $A$ if, whenever $v\in e\in E$, we have $u\in e$. In other words, $(u,v)$ is an arc of $D(\Omega)$ if any edge containing $v$ contains $u$ as well.

For an arc $(u,v)\in A$, $v$ is an {\em outneighbour} of $u$, and $u$ is an {\em inneighbour} of $v$. The set of all outneighbours of $u$ is denoted by $N^{+}_{D(\Omega)}(u)$.
Let $N_{D(\Omega)}^+(X)=\big(\bigcup_{u\in X}N^{+}_{D(\Omega)}(u)\big)\setminus X$; that is, the subset of vertices, not in $X$, being heads of arcs in $A$ having tail in $X$.
The {\em outneighbours} of a set $X$ are the elements of $N_{D(\Omega)}^+(X)$. Note that
 $D(\Omega)$ is a transitive directed graph: if $(u,v)$ and $(v,w)$ with $w\neq u$ are arcs of $D(\Omega)$, then $(u,w)$ is an arc of $D(\Omega)$. In particular, it implies that there is always a nonempty subset $X$ of vertices without outneighbour inducing a complete subgraph in $D(\Omega)$.
Moreover, a vertex of $D(\Omega)$ cannot have two distinct inneighbours in the same $V_i$.

\section{Combinatorial properties of octahedral systems}\label{OSprop}

This section presents properties of octahedral systems generalizing earlier results holding for $n=|V_1|=\ldots=|V_{n}|=d+1$. While Proposition~\ref{parity} and Proposition~\ref{isolated} deal with octahedral systems possibly with isolated vertices, Propositions~\ref{prop:nonisolated},~\ref{prop:upperbound_gen},~\ref{prop:222m}, and~\ref{prop:upperbound} deal with octahedral systems without isolated vertex.


\begin{proposition}\label{parity}
An octahedral system $\Omega=(V_1,\ldots,V_n,E)$ with even $|V_i|$ for $i=1,\dots, n$ has an even number of edges.
\end{proposition}

This proposition provides an alternate definition for octahedral systems where the condition ``$|X\cap V_i|=2$'' is replaced by ``$|X\cap V_i|$ is even'' for $i=1,\ldots,n$.

\begin{proof}
Let $\Xi$ be the set $\{X\subseteq\bigcup_{i=1}^nV_i:\,|X\cap V_i|=2\}$. Since $\Omega$ satisfies the parity property, $|E[X]|$ is even for any  $X\in\Xi$, and $\sum_{X\in\Xi}|E[X]|$ is even. Each edge of $\Omega$ being counted $(|V_1|-1)(|V_2|-1)\ldots(|V_n|-1)$ times in the sum, we have $\sum_{X\in\Xi}|E[X]|= (|V_1|-1)\ldots(|V_n|-1)|E|$. As $(|V_1|-1)\ldots(|V_n|-1)$ is odd,  the number $|E|$ of edges in $\Omega$ is even.
\qquad\end{proof} 

\begin{proposition}\label{isolated}
A non-trivial octahedral system has at least $\min_i |V_i|$ edges. 
\end{proposition}

\begin{proof} Assume without loss of generality that $V_1$ has the smallest cardinality. If no vertex of $V_1$ is isolated, the octahedral system has at least $|V_1|$ edges.
Otherwise, at least one vertex $x$ of $V_1$ is isolated and the parity property applied to an edge, $(|V_1|-1)$  disjoint $1$-transversals, and $x$ gives at least $|V_1|$ edges.
The bound is tight as a $1$-transversal forming an edge with each vertex of $V_1$ is an octahedral system with $|V_1|$ edges.
\qquad\end{proof}

Setting $n=|V_1|=\ldots=|V_{n}|=d+1$ in Proposition~~\ref{parity} and Proposition~\ref{isolated} yields results given in~\cite{CDSX11}.

\begin{proposition}\label{prop:nonisolated}
An octahedral system without isolated vertex has at least $\max_{i\neq j}(|V_i|+|V_j|)-2$ edges.
\end{proposition}

The special case for octahedral systems arising from colourful point configurations, i.e. $\mu(d)\geq 2d$, has be proven in~\cite{DHST06}.

\begin{proof}
Assume without loss of generality that $2\leq|V_1|\leq\ldots\leq|V_{n-1}|\leq|V_n|$. Let $v^*$ be the vertex minimizing the degree in $\Omega$ over $V_n$. If $\deg(v^*)\geq 2$, then there are at least $2|V_n|\geq |V_n|+|V_{n-1}|-2$ edges. Otherwise,  $\deg(v^*)=1$ and we note $e(v^*)$ the unique edge containing $v^*$. Pick $w_i$ in $V_i\setminus e(v^*)$ for all $i<n$. Applying the octahedral property to the transversal $\{w_1,\ldots,w_{n-1}\}$, $e(v^*)$, and any $w\in V_n\setminus\{v^*\}$ yields at least $|V_n|$ edges not intersecting with $V_{n-1}\setminus \left(e(v^*)\cup\{w_{n-1}\}\right)$. Additional $|V_{n-1}|-2$ edges are needed to cover the vertices in $V_{n-1}\setminus \left(e(v^*)\cup\{w_{n-1}\}\right)$. In total we have at least $|V_n|+|V_{n-1}|-2$ edges.
\qquad\end{proof}

\begin{proposition}\label{prop:upperbound_gen}
$\nu(m_1,\ldots,m_n)\leq 2+\sum_{i=1}^n (m_i-2)$. 
\end{proposition}
\begin{proof} For all $(m_1,\ldots,m_n)$, we construct an octahedral system $\Omega^{(m_1,\ldots,m_n)}=(V_1,\ldots,V_n,E^{(m_1,\ldots,m_n)})$ without isolated vertex and with $|V_i|=m_i$, such that $$|E^{(m_1,\ldots,m_n)}|=2+\sum_{i=1}^n (m_i-2).$$ Starting from $\Omega^{(m_1)}$, we inductively build $\Omega^{(m_1,\ldots,m_{n+1})}$ from  $\Omega^{(m_1,\ldots,m_n)}$.

The unique octahedral system without isolated vertex with $n=1$ and $|V_1|=m_1$ is $\Omega^{(m_1)}=(V_1,E^{(m_1)})$ where $E^{(m_1)}=\left\{\{v\}:\, v\in V_1\right\}$.
Assuming that $\Omega^{(m_1,\ldots,m_n)}=(V_1,\ldots,V_n,E^{(m_1,\ldots,m_n)})$ with $|E^{(m_1,\ldots,m_n)}|=2+\sum_{i=1}^n (m_i-2)$ has been built, we build the octahedral system
$\Omega^{(m_1,\ldots,m_{n+1})}=(V_1,\ldots,V_n,V_{n+1},E^{(m_1,\ldots,m_{n+1})})$ by picking an edge $e_1$ in $E^{(m_1,\ldots,m_n)}$ and setting
$$E^{(m_1,\ldots,m_{n+1})}=\{e_1\cup\{u_i\}:\,i=1,\ldots,m_{n+1}-1\}\cup\{e\cup\{u_{m_{n+1}}\}:\,e\in E^{(m_1,\ldots,m_n)}\setminus\{e_1\}\}$$
where $u_1,\ldots,u_{m_{n+1}}$ are the vertices of $V_{n+1}$. Clearly, $|E^{(m_1,\ldots,m_{n+1})}|=m_{n+1}-1+|E^{(m_1,\ldots,m_n)}|-1$; that is, 
$|E^{(m_1,\ldots,m_{n+1})}|=2+\sum_{i=1}^{n+1} (m_i-2)$.
Each vertex of $\Omega^{(m_1,\ldots,m_{n+1})}$ belongs to at least one edge by construction and we need to check the parity condition. 
Let $X\subseteq\bigcup_{i=1}^nV_i$ such that $|X\cap V_i|=2$ for $i=1,\ldots,n+1$ and consider the following four cases:\\

\noindent 
{\em Case $(a)$:} $X\cap V_{n+1}=\{u_j,u_k\}$ with $j\neq m_{n+1}$ and $k\neq m_{n+1}$, and $e_1\subseteq X$. Then, $e_1\cup\{u_j\}$ and $e_1\cup\{u_k\}$ are the only 2 edges induced by $X$ in $\Omega^{(m_1,\ldots,m_{n+1})}$.\\

\noindent 
{\em Case $(b)$:}
$X\cap V_{n+1}=\{u_j,u_k\}$ with $j\neq m_{n+1}$ and $k\neq m_{n+1}$, and $e_1\not\subseteq X$. Then, no edge are induced by $X$ in $\Omega^{(m_1,\ldots,m_{n+1})}$.\\

\noindent 
{\em Case $(c)$:}  $X\cap V_{n+1}=\{u_j,u_{m_{n+1}}\}$ and $e_1\subseteq X$. 
Then, the number of edges in $E^{(m_1,\ldots,m_n)}\setminus\{e_1\}$ induced by $X$ in $\Omega^{(m_1,\ldots,m_n)}$ is odd by the parity property. Hence, the number of edges in $\{e\cup\{u_{m_{n+1}}\}:\,e\in E^{(m_1,\ldots,m_n))}\setminus\{e_1\}\}$ induced by $X$ in $\Omega^{(m_1,\ldots,m_{n+1})}$ is odd as well. These edges, along with the edge $e_1\cup\{u_j\}$, are the only edges induced by $X$ in $\Omega^{(m_1,\ldots,m_{n+1})}$, i.e. the parity condition holds.\\

\noindent 
{\em Case $(d)$:}  $X\cap V_{n+1}=\{u_j,u_{m_{n+1}}\}$ and $e_1\not\subseteq X$.
Then, the number of edges in $E^{(m_1,\ldots,m_n)}\setminus\{e_1\}$ induced by $X$ in $\Omega^{(m_1,\ldots,m_n)}$ is even by the parity property. Hence, the number of edges in $\{e\cup\{u_{m_{n+1}}\}:\,e\in E^{(m_1,\ldots,m_n)}\setminus\{e_1\}\}$ induced by $X$ in $\Omega^{(m_1,\ldots,m_{n+1})}$ is even as well. These edges are the only edges induced by $X$ in $\Omega^{(m_1,\ldots,m_{n+1})}$, i.e. the parity condition holds.
\qquad\end{proof}

Figures~\ref{fig:omega33} and~\ref{fig:omega34} illustrate the construction in the proof of Proposition~\ref{prop:upperbound_gen} for $n=m_1=m_2=m_3=3$, and for $n-1=m_1=m_2=m_3=m_4=3$.

\begin{figure}
\begin{center}
\begin{tikzpicture}
    \node (v1) at (0,4) {};
    \node (v2) at (3,4) {};
    \node (v3) at (6,4) {};
    \node (v4) at (0,2) {};
    \node (v5) at (3,2) {};
    \node (v6) at (6,2) {};
    \node (v7) at (0,0) {};
    \node (v8) at (3,0) {};
    \node (v9) at (6,0) {};

    \begin{scope}[fill opacity=0.5]
    \filldraw[] ($(v1)+(-0.2,0)$) 
    	to[out=90,in=180] ($(v1) + (0,0.2)$) 
        to[out=0,in=180] ($(v2) + (0,0.2)$) 
        to[out=0,in=180] ($(v3) + (0,0.2)$)
        to[out=0,in=90] ($(v3) + (0.2,0)$)
        to[out=270,in=0] ($(v3) + (0,-0.2)$) 
        to[out=180,in=0] ($(v2) + (0,-0.2)$)
        to[out=180,in=0] ($(v1) + (0,-0.2)$) 
        to[out=180,in=270] ($(v1)+(-0.2,0)$);
    \filldraw[] ($(v1)+(-0.2,0)$) 
    	to[out=90,in=180] ($(v1) + (0,0.2)$) 
        to[out=0,in=180] ($(v2) + (0,0.2)$) 
        to[out=0,in=180] ($(v6) + (0,0.2)$)
        to[out=0,in=90] ($(v6) + (0.2,0)$)
        to[out=270,in=0] ($(v6) + (0,-0.2)$)
        to[out=180,in=0] ($(v2) + (0,-0.2)$)
        to[out=180,in=0] ($(v1) + (0,-0.2)$) 
        to[out=180,in=270] ($(v1)+(-0.2,0)$);
    \filldraw[] ($(v1)+(-0.2,0)$) 
    	to[out=90,in=180] ($(v1) + (0,0.2)$)
        to[out=0,in=135] ($(v5) + (0,0.5)$) 
        to[out=315,in=180] ($(v9) + (0,0.2)$)
        to[out=0,in=90] ($(v9) + (0.2,0)$)
        to[out=270,in=315] ($(v9) + (0,-0.2)$)
        to[out=180,in=315] ($(v5) + (0,-0.5)$)
        to[out=135,in=0] ($(v1) + (0,-0.2)$) 
        to[out=180,in=270] ($(v1)+(-0.2,0)$);
    \filldraw[] ($(v4)+(-0.2,0)$) 
        to[out=90,in=180] ($(v4) + (0,0.2)$)
        to[out=0,in=180] ($(v8) + (0,0.2)$)
        to[out=0,in=180] ($(v9) + (0,0.2)$)
        to[out=0,in=90] ($(v9) + (0.2,0)$)
        to[out=270,in=0] ($(v9) + (0,-0.2)$)
        to[out=180,in=0] ($(v8) + (0,-0.2)$)
        to[out=180,in=0] ($(v4) + (0,-0.2)$)
        to[out=180,in=270] ($(v4)+(-0.2,0)$);
    \filldraw[] ($(v7)+(-0.2,0)$) 
    	to[out=90,in=180] ($(v7) + (0,0.2)$) 
        to[out=0,in=180] ($(v8) + (0,0.2)$) 
        to[out=0,in=180] ($(v9) + (0,0.2)$)
        to[out=0,in=90] ($(v9) + (0.2,0)$)
        to[out=270,in=0] ($(v9) + (0,-0.2)$) 
        to[out=180,in=0] ($(v8) + (0,-0.2)$)
        to[out=180,in=0] ($(v7) + (0,-0.2)$) 
        to[out=180,in=270] ($(v7)+(-0.2,0)$);
    \end{scope}

    \foreach \v in {1,2,...,9} {
        \fill (v\v) circle (0.1);
        
    }
    
            \node at (5,4) {$e_1$};
\end{tikzpicture}
\caption{$\Omega^{(3,3)}$: a $(3,3,3)$-octahedral system matching the upper bound given in Proposition~\ref{prop:upperbound_gen}}
\label{fig:omega33}
\end{center}
\end{figure}

 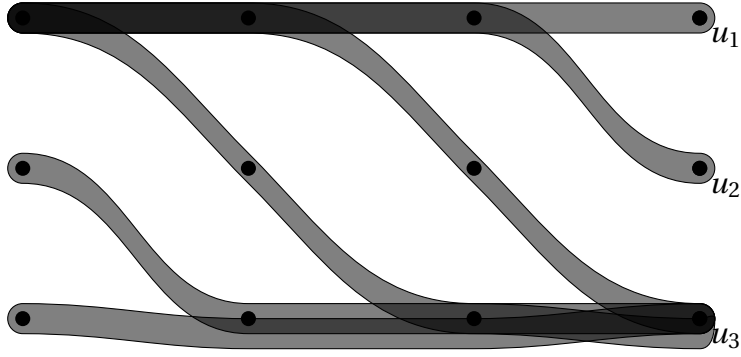
\begin{figure}
\begin{center}
\begin{tikzpicture}
    \node (v1) at (0,4) {};
    \node (v2) at (3,4) {};
    \node (v3) at (6,4) {};
    \node (v4) at (0,2) {};
    \node (v5) at (3,2) {};
    \node (v6) at (6,2) {};
    \node (v7) at (0,0) {};
    \node (v8) at (3,0) {};
    \node (v9) at (6,0) {};
        \node (v10) at (9,4) {};
        \node (v11) at (9,2) {};
        \node (v12) at (9,0) {};

    \begin{scope}[fill opacity=0.5]
    \filldraw[] ($(v1)+(-0.2,0)$) 
    	to[out=90,in=180] ($(v1) + (0,0.2)$) 
        to[out=0,in=180] ($(v2) + (0,0.2)$) 
        to[out=0,in=180] ($(v3) + (0,0.2)$)
        to[out=0,in=180] ($(v10) + (0,0.2)$) 
        to[out=0,in=90] ($(v10) + (0.2,0)$)
        to[out=270,in=0] ($(v10) + (0,-0.2)$) 
        to[out=180,in=0] ($(v3) + (0,-0.2)$)
        to[out=180,in=0] ($(v2) + (0,-0.2)$)
        to[out=180,in=0] ($(v1) + (0,-0.2)$) 
        to[out=180,in=270] ($(v1)+(-0.2,0)$);
    \filldraw[] ($(v1)+(-0.2,0)$) 
    	to[out=90,in=180] ($(v1) + (0,0.2)$) 
        to[out=0,in=180] ($(v2) + (0,0.2)$) 
        to[out=0,in=180] ($(v3) + (0,0.2)$)
        to[out=0,in=180] ($(v11) + (0,0.2)$) 
        to[out=0,in=90] ($(v11) + (0.2,0)$)
        to[out=270,in=0] ($(v11) + (0,-0.2)$) 
        to[out=180,in=0] ($(v3) + (0,-0.2)$)
        to[out=180,in=0] ($(v2) + (0,-0.2)$)
        to[out=180,in=0] ($(v1) + (0,-0.2)$) 
        to[out=180,in=270] ($(v1)+(-0.2,0)$);
    \filldraw[] ($(v1)+(-0.2,0)$) 
    	to[out=90,in=180] ($(v1) + (0,0.2)$) 
        to[out=0,in=180] ($(v2) + (0,0.2)$) 
        to[out=0,in=135] ($(v6) + (0,0.2)$)
		to[out=315,in=180] ($(v12) + (0,0.2)$) 
        to[out=0,in=90] ($(v12) + (0.2,0)$)
        to[out=270,in=0] ($(v12) + (0,-0.2)$) 
        to[out=180,in=315] ($(v6) + (0,-0.2)$)
        to[out=135,in=0] ($(v2) + (0,-0.2)$)
        to[out=180,in=0] ($(v1) + (0,-0.2)$) 
        to[out=180,in=270] ($(v1)+(-0.2,0)$);
    \filldraw[] ($(v1)+(-0.2,0)$) 
    	to[out=90,in=180] ($(v1) + (0,0.2)$)
        to[out=0,in=135] ($(v5) + (0,0.2)$) 
        to[out=315,in=180] ($(v9) + (0,0.2)$)
        to[out=0,in=180] ($(v12) + (0,0.2)$) 
        to[out=0,in=90] ($(v12) + (0.2,0)$)
        to[out=270,in=0] ($(v12) + (0,-0.2)$)
        to[out=180,in=0] ($(v9) + (0,-0.2)$)
        to[out=180,in=315] ($(v5) + (0,-0.2)$)
        to[out=135,in=0] ($(v1) + (0,-0.2)$) 
        to[out=180,in=270] ($(v1)+(-0.2,0)$);
    \filldraw[] ($(v4)+(-0.2,0)$) 
        to[out=90,in=180] ($(v4) + (0,0.2)$)
        to[out=0,in=180] ($(v8) + (0,0.2)$)
        to[out=0,in=180] ($(v9) + (0,0.2)$)
        to[out=0,in=180] ($(v12) + (0,0)$) 
        to[out=0,in=90] ($(v12) + (0.2,0)$)
        to[out=270,in=0] ($(v12) + (0,-0.4)$)
        to[out=180,in=0] ($(v9) + (0,-0.2)$)
        to[out=180,in=0] ($(v8) + (0,-0.2)$)
        to[out=180,in=0] ($(v4) + (0,-0.2)$)
        to[out=180,in=270] ($(v4)+(-0.2,0)$);
    \filldraw[] ($(v7)+(-0.2,0)$) 
    	to[out=90,in=180] ($(v7) + (0,0.2)$) 
        to[out=0,in=180] ($(v8) + (0,0)$) 
        to[out=0,in=180] ($(v9) + (0,0)$)
        to[out=0,in=180] ($(v12) + (0,0.2)$) 
        to[out=0,in=90] ($(v12) + (0.2,0)$)
        to[out=270,in=0] ($(v12) + (0,-0.2)$) 
        to[out=180,in=0] ($(v9) + (0,-0.4)$)
        to[out=180,in=0] ($(v8) + (0,-0.4)$)
        to[out=180,in=0] ($(v7) + (0,-0.2)$) 
        to[out=180,in=270] ($(v7)+(-0.2,0)$);
    \end{scope}

    \foreach \v in {1,2,...,12} {
        \fill (v\v) circle (0.1);
    }
    
    \fill (v10) circle (0.1) node [below right] {$u_1$};
    \fill (v11) circle (0.1) node [below right] {$u_2$};
    \fill (v12) circle (0.1) node [below right] {$u_3$};

\end{tikzpicture}
\caption{$\Omega^{(3,4)}$: a $(3,3,3,3)$-octahedral system matching the upper bound given in Proposition~\ref{prop:upperbound_gen}}
\label{fig:omega34}
\end{center}
\end{figure}

Proposition~\ref{prop:upperbound_gen} combined with Proposition~\ref{prop:nonisolated} directly implies Proposition~\ref{prop:222m} given without proof.
\begin{proposition}\label{prop:222m}
$\nu(2,\ldots,2,m_{n-1},m_n)=m_{n-1}+m_n-2$ for $m_{n-1},m_{n}\geq 2$.
\end{proposition}

When all $m_i$ are equal, the bound given in Proposition~\ref{prop:upperbound_gen} can be improved.
\begin{proposition}\label{prop:upperbound}
$\nu(\overbrace{m,\ldots,m}^{n\mbox{ \scriptsize{\rm{times}}}})\leq\min(m^2,n(m-2)+2)$ for all $m,n\geq 1$. 
\end{proposition}
\begin{proof}
We construct an $(m,\dots,m)$-octahedral system without isolated vertex and with $m^2$ edges. Consider $m$ disjoint $n$-transversals, and form $m$ edges from each of these $n$-transversals by adding a distinct vertex of $V_n$. We obtain an octahedral system without isolated vertex with $m^2$ edges. The other inequality is a corollary of Proposition~\ref{prop:upperbound_gen}.
\qquad\end{proof}

Propositions~\ref{prop:upperbound_gen} and~\ref{prop:upperbound} can be seen as combinatorial counterparts and generalizations of $\mu(d)\leq d^2+1$ proven in~\cite{DHST06}.

An approach similar to the one developed in Section~\ref{mu4} shows that 
$$\nu(\overbrace{2,\ldots,2}^{z\mbox{ \scriptsize{times}}},\overbrace{3,\ldots,3}^{4-z\mbox{ \scriptsize{times}}},4)= 8-z\mbox{ and } 
\nu(\overbrace{3,\ldots,3}^{z\mbox{ \scriptsize{times}}},\overbrace{4,\ldots,4}^{5-z\mbox{ \scriptsize{times}}})= 12-z\mbox{ for }z=0,\ldots,4.$$
In other words,  the inequality given in Proposition~\ref{prop:upperbound_gen} holds with equality for small $m_i$'s and $n$ at most $5$. While this inequality also holds with equality for any $n$ when $m_1=\ldots=m_{n-2}=2$ by Proposition~\ref{prop:222m}, the inequality can be strict as, for example, $\nu(3,\ldots,3)<2+n$ for $n\geq 8$ by Proposition~\ref{prop:upperbound}.

\section{Additional results}\label{Q&A}
This section provides answers to open questions raised in~\cite{CDSX11} by
determining the number of distinct octahedral systems, and by showing that some octahedral systems 
cannot arise from a colourful point configuration.

Proposition~\ref{prop:diff} shows that the set of all octahedral systems defined on the same $V_i$'s
equipped with the symmetric difference as addition is a $\mathbb{F}_2$ vector space.

\begin{proposition}\label{prop:diff}
Let $\Omega_1=(V_1,\ldots,V_n,E_1)$ and $\Omega_2=(V_1,\ldots,V_n,E_2)$ be two octahedral systems on the same sets of vertices, the symmetric difference $\Omega_1\triangle\Omega_2=(V_1,\ldots,V_n,E_1\triangle E_2)$ is an octahedral system.
\end{proposition}
\begin{proof}
Consider $X\subseteq\bigcup_{i=1}^nV_i$ such that $|X\cap V_i|=2$ for $i=1,\ldots,n$. We have
$|(E_1\triangle E_2)[X]|=|E_1[X]|+|E_2[X]|-2|(E_1\cap E_2)[X]|$, and therefore the parity condition holds for $\Omega_1\triangle\Omega_2$.
\qquad\end{proof}

Proposition~\ref{prop:diff} can be used to build octahedral systems or to prove the non-existence of others. For instance, Proposition~\ref{prop:diff} implies that there is a $(3,3,3)$-octahedral system without isolated vertex with exactly $22$ edges by setting $\Omega_1$ to be the complete $(3,3,3)$-octahedral system with $27$ edges, and $\Omega_2$ to be the $(3,3,3)$-octahedral system with exactly $5$ edges given in Figure~\ref{fig:omega9}. The octahedral system $\Omega_1\triangle\Omega_2$ is without isolated vertex since each vertex in $\Omega_1$ is of degree $9$. Similarly, Proposition~\ref{prop:diff} shows that no $(3,3,3)$-octahedral system with exactly $25$ or $26$ edges exists. Otherwise a $(3,3,3)$-octahedral system with exactly $1$ or $2$ edges would exist, contradicting Proposition~\ref{isolated}.

\begin{theorem}\label{thm:number}
Given $n$ disjoint finite vertex sets $V_1,\ldots,V_n$, the number of octahedral systems on $V_1,\ldots,V_n$ is $2^{\Pi_{i=1}^n|V_i|-\Pi_{i=1}^n(|V_i|-1)}$.
\end{theorem}
\begin{proof} 
We denote by $F_i$ the binary vector space $\mathbb{F}_2^{V_i}$ and by $G_i$ its subspace whose vectors have an even number of $1$. Let $\mathcal{H}$ be the tensor product $F_1\otimes\ldots\otimes F_n$ and $\mathcal{X}$ its subspace $G_1\otimes\ldots\otimes G_n$. Note that there is a bijection between the elements of $\mathcal{H}$ and the $n$-partite hypergraphs on vertex sets $V_1,\ldots,V_n$: each edge $\{v_1,\ldots,v_n\}$ of such a hypergraph $H$, with $v_i\in V_i$ for all $i$, is identified with the vector $x_1\otimes\ldots\otimes x_n$, where $x_i$ is the unit vector of $F_i$ having a $1$ uniquely at position $v_i$. 
Define now $\psi$ as follows.
$$\begin{array}{cccc}\psi: & \mathcal{H} & \rightarrow & \mathcal{X}^* \\
& H & \mapsto & \langle H,\cdot\rangle\end{array}$$ 
By the above identification and according to the alternate definition of an octahedral system given by Proposition~\ref{parity},
the subspace $\ker\psi$ of $\mathcal{H}$ is  the set of all octahedral systems on vertex sets $V_1,\ldots,V_n$. Note that by definition $\psi$ is surjective. Therefore, we have $\dim\ker\psi+\dim\mathcal{X}^*=\dim\mathcal{H}$ which implies  $\dim\ker\psi=\dim\mathcal{H}-\dim\mathcal{X}$ using the isomorphism between a vector space and its dual. The dimension of $\mathcal{H}$ is $\Pi_{i=1}^n|V_i|$ and the dimension of $\mathcal{X}$ is $\Pi_{i=1}^n(|V_i|-1)$. This leads to the desired conclusion.
 \qquad\end{proof}

The vector space structure of the kernel of $\psi$ gives another proof of Proposition~\ref{prop:diff}. 
Karasev~\cite{Kar12b} noted that the set of all colourful simplices in a colourful point configuration forms a $d$-dimensional coboundary of the join $\mathbf{S}_1*\ldots*\mathbf{S}_{d+1}$ with mod $2$ coefficients. We further note that the octahedral systems form precisely the $(n-1)$-coboundaries of the join $V_1*\ldots*V_n$ with mod $2$ coefficients. Indeed, with mod $2$ coefficients and with the notations of the proof of Theorem~\ref{thm:number}, the coboundaries of $V_1*\ldots*V_n$ are generated by the vectors of the form $x_1\otimes\ldots\otimes x_{j-1}\otimes e\otimes x_{j+1}\otimes\ldots\otimes x_n$, with $j\in\{1,\ldots,n\}$ and $e=(1,\ldots,1)\in\mathbb{F}_2^{V_j}$. With the identification given in Proposition~\ref{prop:diff}, each of these vectors forms an octahedral system. Moreover, they generate the set of all octahedral systems seen as a vector space: 
given an octahedral system and one of its vertex $v$ of nonzero degree, we can add vectors of the above form in order to make $v$ isolated; repeating this argument for each $V_i$, we get an octahedral system with an isolated vertex in each $V_i$; such an octahedral system is empty, i.e. the zero vector of the space of octahedral systems.

A natural question is whether there is a {\em non-labelled} version of Theorem~\ref{thm:number}, that is whether it is possible to compute, or to bound, the number of non-isomorphic octahedral systems. Two isomorphic octahedral systems; that is, identical up to a permutation of the $V_i$'s, or of the vertices in one of the $V_i$'s, are considered distinct in Theorem~\ref{thm:number}. Answering this question would fully answer Question 7 of~\cite{CDSX11}. While Polya's theory might be helpful, we were not able to address the question. 

Finally, Question 6 of~\cite{CDSX11} asks whether any octahedral system $\Omega=(V_1,\ldots,V_n,E)$ with $n=|V_1|=\ldots=|V_n|=d+1$ can arise from a colourful point configuration $\mathbf{S}_1,\ldots,\mathbf{S}_{d+1}$ in $\mathbb{R}^d$? That is, are all octahedral systems {\em realisable}? We answer negatively this question in Proposition~\ref{prop:nonreal}.

\begin{proposition}\label{prop:nonreal}
Not all octahedral systems are realisable.
\end{proposition}

This statement is also true for octahedral systems without isolated vertex.

\begin{proof}
We provide an example of a non-realisable octahedral system without isolated vertex in Figure~\ref{fig:omega9}. Indeed, suppose by contradiction that this octahedral system can be realized as a colourful point configuration $\mathbf{S}_1,\mathbf{S}_2,\mathbf{S}_3$. Without loss of generality, we can assume that all the points lie on a circle centred at $\zero$. Take $x_3\in\mathbf{S}_3$, and consider the line $\ell$ going through $x_3$ and $\zero$. There are at least two points $x_1$ and $x'_1$ of $\mathbf{S}_1$ on the same side of $\ell$. There is a point $x_2\in\mathbf{S}_2$, respectively $x'_2\in\mathbf{S}_2$, on the other side of the line $\ell$ such that 
 $\zero\in\conv(x_1,x_2,x_3)$, respectively $\zero\in\conv(x'_1,x'_2,x_3)$. Assume without loss of generality that $x'_2$ is further away from $x_3$ than $x_2$. Then, $\conv(x_1,x'_2,x_3)$ contains $\zero$ as well, contradicting the definition of the octahedral system given in Figure~\ref{fig:omega9}.
 \qquad\end{proof}

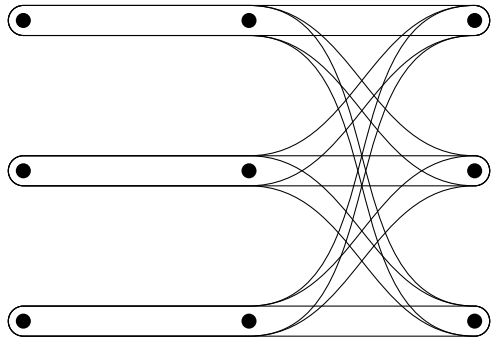
\begin{figure}[ht]
\begin{center}
\begin{tikzpicture}
    \node (v1) at (0,4) {};
    \node (v2) at (3,4) {};
    \node (v3) at (6,4) {};
    \node (v4) at (0,2) {};
    \node (v5) at (3,2) {};
    \node (v6) at (6,2) {};
    \node (v7) at (0,0) {};
    \node (v8) at (3,0) {};
    \node (v9) at (6,0) {};

    \begin{scope}[fill opacity=0.5]
    \draw ($(v1)+(-0.2,0)$) 
    	to[out=90,in=180] ($(v1) + (0,0.2)$) 
        to[out=0,in=180] ($(v2) + (0,0.2)$) 
        to[out=0,in=180] ($(v3) + (0,0.2)$)
        to[out=0,in=90] ($(v3) + (0.2,0)$)
        to[out=270,in=0] ($(v3) + (0,-0.2)$) 
        to[out=180,in=0] ($(v2) + (0,-0.2)$)
        to[out=180,in=0] ($(v1) + (0,-0.2)$) 
        to[out=180,in=270] ($(v1)+(-0.2,0)$);
    \draw ($(v1)+(-0.2,0)$) 
    	to[out=90,in=180] ($(v1) + (0,0.2)$) 
        to[out=0,in=180] ($(v2) + (0,0.2)$) 
        to[out=0,in=180] ($(v6) + (0,0.2)$)
        to[out=0,in=90] ($(v6) + (0.2,0)$)
        to[out=270,in=0] ($(v6) + (0,-0.2)$)
        to[out=180,in=0] ($(v2) + (0,-0.2)$)
        to[out=180,in=0] ($(v1) + (0,-0.2)$) 
        to[out=180,in=270] ($(v1)+(-0.2,0)$);
    \draw ($(v1)+(-0.2,0)$) 
    	to[out=90,in=180] ($(v1) + (0,0.2)$)
        to[out=0,in=180] ($(v2) + (0,0.2)$) 
        to[out=0,in=180] ($(v9) + (0,0.2)$)
        to[out=0,in=90] ($(v9) + (0.2,0)$)
        to[out=270,in=0] ($(v9) + (0,-0.2)$)
        to[out=180,in=0] ($(v2) + (0,-0.2)$)
        to[out=180,in=0] ($(v1) + (0,-0.2)$) 
        to[out=180,in=270] ($(v1)+(-0.2,0)$);
    \draw ($(v4)+(-0.2,0)$) 
    	to[out=90,in=180] ($(v4) + (0,0.2)$) 
        to[out=0,in=180] ($(v5) + (0,0.2)$) 
        to[out=0,in=180] ($(v3) + (0,0.2)$)
        to[out=0,in=90] ($(v3) + (0.2,0)$)
        to[out=270,in=0] ($(v3) + (0,-0.2)$) 
        to[out=180,in=0] ($(v5) + (0,-0.2)$)
        to[out=180,in=0] ($(v4) + (0,-0.2)$) 
        to[out=180,in=270] ($(v4)+(-0.2,0)$);
    \draw($(v4)+(-0.2,0)$) 
    	to[out=90,in=180] ($(v4) + (0,0.2)$) 
        to[out=0,in=180] ($(v5) + (0,0.2)$) 
        to[out=0,in=180] ($(v6) + (0,0.2)$)
        to[out=0,in=90] ($(v6) + (0.2,0)$)
        to[out=270,in=0] ($(v6) + (0,-0.2)$)
        to[out=180,in=0] ($(v5) + (0,-0.2)$)
        to[out=180,in=0] ($(v4) + (0,-0.2)$) 
        to[out=180,in=270] ($(v4)+(-0.2,0)$);
    \draw ($(v4)+(-0.2,0)$) 
    	to[out=90,in=180] ($(v4) + (0,0.2)$)
        to[out=0,in=180] ($(v5) + (0,0.2)$) 
        to[out=0,in=180] ($(v9) + (0,0.2)$)
        to[out=0,in=90] ($(v9) + (0.2,0)$)
        to[out=270,in=0] ($(v9) + (0,-0.2)$)
        to[out=180,in=0] ($(v5) + (0,-0.2)$)
        to[out=180,in=0] ($(v4) + (0,-0.2)$) 
        to[out=180,in=270] ($(v4)+(-0.2,0)$);
    \draw ($(v7)+(-0.2,0)$) 
    	to[out=90,in=180] ($(v7) + (0,0.2)$) 
        to[out=0,in=180] ($(v8) + (0,0.2)$) 
        to[out=0,in=180] ($(v3) + (0,0.2)$)
        to[out=0,in=90] ($(v3) + (0.2,0)$)
        to[out=270,in=0] ($(v3) + (0,-0.2)$) 
        to[out=180,in=0] ($(v8) + (0,-0.2)$)
        to[out=180,in=0] ($(v7) + (0,-0.2)$) 
        to[out=180,in=270] ($(v7)+(-0.2,0)$);
    \draw ($(v7)+(-0.2,0)$) 
    	to[out=90,in=180] ($(v7) + (0,0.2)$) 
        to[out=0,in=180] ($(v8) + (0,0.2)$) 
        to[out=0,in=180] ($(v6) + (0,0.2)$)
        to[out=0,in=90] ($(v6) + (0.2,0)$)
        to[out=270,in=0] ($(v6) + (0,-0.2)$)
        to[out=180,in=0] ($(v8) + (0,-0.2)$)
        to[out=180,in=0] ($(v7) + (0,-0.2)$) 
        to[out=180,in=270] ($(v7)+(-0.2,0)$);
    \draw ($(v7)+(-0.2,0)$) 
    	to[out=90,in=180] ($(v7) + (0,0.2)$)
        to[out=0,in=180] ($(v8) + (0,0.2)$) 
        to[out=0,in=180] ($(v9) + (0,0.2)$)
        to[out=0,in=90] ($(v9) + (0.2,0)$)
        to[out=270,in=0] ($(v9) + (0,-0.2)$)
        to[out=180,in=0] ($(v8) + (0,-0.2)$)
        to[out=180,in=0] ($(v7) + (0,-0.2)$) 
        to[out=180,in=270] ($(v7)+(-0.2,0)$);
    \end{scope}

    \foreach \v in {1,2,...,9} {
        \fill (v\v) circle (0.1);
        
    }

\end{tikzpicture}
\caption{A non realisable $(3,3,3)$-octahedral system with $9$ edges}
\label{fig:omega9}
\end{center}
\end{figure}
We conclude the section with a question to which the intuitive answer is yes but we are unable to settle.

\begin{question}
Is $\nu(m_1,\ldots,m_n)$ increasing with each of the $m_i$?
\end{question}

\section{Proof of the main result}\label{proofs}
\subsection{Technical lemmas}
While Lemma~\ref{induc} allows induction within octahedral systems, Lemmas~\ref{tech},~\ref{lem:1in1}, and~\ref{lem:2in1} are used in the subsequent sections to bound the number of edges of an octahedral system without isolated vertex. 
\begin{lemma}\label{induc}
Consider an octahedral system $\Omega$ without isolated vertex.
If $X$ induces a complete subgraph in $D(\Omega)$ and $N_{D(\Omega)}^+(X)=\emptyset$, then $\Omega\setminus X$ is an octahedral system without isolated vertex. 
\end{lemma}
\begin{proof}
The parity condition is clearly satisfied for $\Omega\setminus X$, and each vertex of $\Omega\setminus X$ is contained in at least one edge.
\qquad\end{proof}

\begin{lemma}\label{tech}
For $n\geq 4$, consider a $(\overbrace{k-1,\ldots,k-1}^{z\mbox{ \scriptsize{\rm{times}}}},\overbrace{k,\ldots,k}^{k-z\mbox{ \scriptsize{\rm{times}}}},m_{k+1},\ldots,m_n)$-octahedral system $\Omega=(V_1,\ldots,V_n,E)$ without isolated vertex, with $3\leq k\leq m_{k+1}\leq\ldots\leq m_n$ and $0\leq z<k\leq n$. If there is a subset $X\subseteq\bigcup^n_{i=z+1} V_i$ of cardinality at least 2 inducing in $D(\Omega)$ a complete subgraph without outneighbour, then $\Omega$ has at least $(k-1)^2+2$ edges, unless $\Omega$ is a $(2,2,3,3)$-octahedral system. Under the same condition on $X$, a $(2,2,3,3)$-octahedral system has at least $5$ edges.

\end{lemma}
\begin{proof} 
Any edge intersecting $X$ contains $X$ since $X$ induces a complete subgraph in $D(\Omega)$, implying $\deg_{\Omega}(X)\geq 1$. Moreover, we have $|X\cap V_i|\leq 1$ for $i=1,\ldots,n$.\\

\noindent 
{\em Case $(a)$:} $\deg_{\Omega}(X)\geq 2$. Choose $i^*$ such that $|X\cap V_{i^*}|\neq 0$. We first note that the degree of each $w$ in $V_{i^*}\setminus X$ is at least $k-1$.

Indeed, take an edge $e$ containing $w$ and a $i^*$-transversal $T$ disjoint from $e$ and $X$. Note that $e$ does not contain any vertex of $X$ as underlined in the first sentence of the proof. Apply the weak form of the parity property to $e$, $T$, and the unique vertex $x$ in $X\cap V_{i^*}$. There is an edge distinct from $e$ in $e\cup T\cup\{x\}$. Note that this edge contains $w$, otherwise it would contain $x$ and any other vertex in $X$. It also contains at least one vertex in $T$. For a fixed $e$, we can actually choose $k-2$ disjoint $i^*$-transversals $T$ of that kind and apply the weak form of the parity property to each of them. Thus, there are $k-2$ distinct edges containing $w$ in addition to $e$. 

Therefore, we have in total at least $(k-1)^2$ edges, in addition to $\deg_{\Omega}(X)\geq 2$ edges.\\

\noindent 
{\em Case $(b)$:} $\deg_{\Omega}(X)=1$. Let $e(X)$ denote the unique edge containing $X$. For each $i$ such that $|X\cap V_i|=0$, pick a vertex $w_i$ in $V_i\setminus e(X)$. Applying the parity property to $e(X)$, the $w_i$'s, and any colourful selection of $u_i\in V_i\setminus X$ such that $|X\cap V_i|\neq 0$ shows that there is at least one additional edge containing all $u_i$'s. We can actually choose $(k-1)^{|X|}$ distinct colourful selections of $u_i$'s. With $e(X)$, there are in total $(k-1)^{|X|}+1$ edges.

If $|X|\geq 3$, then $(k-1)^{|X|}+1\geq(k-1)^2+2$. If $|X|=2$, there exists $j\geq n-2$ such that $|X\cap V_j|=0$. If $|V_j|\geq 3$, then at least $|V_j|-2\geq 1$ edges are needed to cover the vertices of $V_j$ not belonging to these $(k-1)^{|X|}+1$ edges. Otherwise, $|V_j|=2$ and we necessarily have $j\leq z$ and $k=3$. In this case, we have thus $k-1\geq z\geq n-2$, i.e. $n=4$ and $z=2$. $\Omega$ is then a $(2,2,3,3)$-octahedral system and $(k-1)^{|X|}+1=5$.\qquad\end{proof}

\noindent While Lemma~\ref{lem:1in1} is similar to Lemma~\ref{tech}, we were not able to find a common generalization.

\begin{lemma}\label{lem:1in1}
Consider a $(\overbrace{k-1,\ldots,k-1}^{z\mbox{ \scriptsize{\rm{times}}}},\overbrace{k,\ldots,k}^{k-z\mbox{ \scriptsize{\rm{times}}}},m_{k+1},\ldots,m_n)$-octahedral system $\Omega=(V_1,\ldots,V_n,E)$ without isolated vertex, with $3\leq k\leq m_{k+1}\leq\ldots\leq m_n$ and $0\leq z<k\leq n$. If there is a subset $X\subseteq\bigcup^n_{i=z+1} V_i$ of cardinality at least 2 inducing in $D(\Omega)$ a complete subgraph without outneighbour, then $\Omega$ has at least $(k-1)^2+|V_{n-1}|+|V_n|-2k+1$ edges.
\end{lemma}

\begin{proof}
Choose $i^*$ such that $X\cap V_{i^*}\neq\emptyset$. Choose $W_{i^*}\subseteq V_{i^*}\setminus X$ of cardinality $k-1$. For each vertex $w\in W_{i^*}$, choose an edge $e(w)$ containing $w$. We can assume that there is a vertex $w^*\in W_{i^*}$ such that $e(w^*)$ contains a vertex $v^*$ minimizing the degree in $\Omega$ over $V_n\setminus X$. Choose $W_i\subseteq V_i$ for $i\neq i^*$ such that $|W_i|=k-1$ and 
$$\bigcup_{w\in W_{i^*}} e(w) \subseteq W=\bigcup_{i=1}^{n} W_i.$$\\

\noindent 
{\em Case $(a)$:} the degree of $v^*$ in $\Omega$ is at most $k-2$. For all $w\in W_{i^*}$, applying the parity property to $e(w)$, the unique vertex of $X\cap V_{i^*}$, and $k-2$ disjoint $i^*$-transversals in $W$ yields $(k-1)^2$ distinct edges, in a similar way as in Case $(a)$ of the proof of Lemma~\ref{tech}. Applying the parity property to $e(w^*)$, any $n$-transversal in $W$ not intersecting the neighbourhood of $v^*$ in $\Omega$, and each vertex in $V_n\setminus W_n$ gives $|V_n|-k+1$ additional edges not intersecting $V_{n-1}\setminus W_{n-1}$. Additional $|V_{n-1}|-k+1$  edges are needed to cover the vertices of  $V_{n-1}\setminus W_{n-1}$. In total we have at least $(k-1)^2+|V_{n}|+|V_{n-1}|-2(k-1)$ edges.\\

\noindent 
{\em Case $(b)$:} the degree of $v^*$ in $\Omega$ is at least $k-1$. We have then at least $(k-1)(|V_{n}|-1)+1= (k-1)^2+(k-1)(|V_{n}|-k)+1\geq (k-1)^2+|V_{n-1}|+|V_{n}|-2k+1$ edges.
\qquad\end{proof}

\begin{lemma}\label{lem:2in1}
Consider a $(\overbrace{k-1,\ldots,k-1}^{z\mbox{ \scriptsize{\rm{times}}}},\overbrace{k,\ldots,k}^{k-z\mbox{ \scriptsize{\rm{times}}}},m_{k+1},\ldots,m_n)$-octahedral system $\Omega=(V_1,\ldots,V_n,E)$ without isolated vertex, with $3\leq k\leq m_{k+1}\leq\ldots\leq m_n$ and $0\leq z<k\leq n$. If there are at least two vertices $v$ and $v'$ of $V_n$ having outneighbours in $D(\Omega)$ in the same $V_{i^*}$ with $i^*<k$, then the octahedral system has at least $|V_{i^*}|(k-1)+|V_{n-1}|+|V_{n}|-2k$ edges.
\end{lemma}

\begin{proof}
Let $u$ and $u'$ be the two vertices in $V_{i^*}$ with $(v,u)$ and $(v',u')$ forming arcs in $D(\Omega)$. Note that according to the basic properties of $D(\Omega)$, we have $u\neq u'$.\\

\noindent 
{\em Case $(a)$:} $|V_{i^*}|=k$. For each vertex $w\in V_{i^*}$, choose an edge $e(w)$ containing $w$. We can assume that there is a vertex $w^*\in V_{i^*}$ such that $e(w^*)$ contains a vertex $v^*$ in $V_n$ of minimal degree in $\Omega$. Choose $W_i\subseteq V_i$ such that $|W_i|=k-1$ for $i=1,\ldots,z$, $|W_i|=k$ for $i=z+1,\ldots,n$, and 
$$\bigcup_{w\in V_{i^*}} e(w) \subseteq W=\bigcup_{i=1}^{n} W_i.$$
\noindent We first show that the degree of any vertex in $V_{i^*}$ is at least $k-1$ in the hypergraph induced by $W$. Pick $w\in V_{i^*}$ and consider $e(w)$. If $v\in e(w)$, take $k-2$ disjoint $i^*$-transversals in $W$ not containing $v'$ and not intersecting with $e(w)$.  Applying the parity property to $e(w)$, $u'$, and each of those $i^*$-transversals yields, in addition to $e(w)$, at least $k-2$ edges containing $w$. Otherwise, take $k-2$ disjoint $i^*$-transversals in $W$ not containing $v$ and not intersecting with $e(w)$, and apply the parity property to $e(w)$, $u$, and each of those $i^*$-transversals. Therefore, in both cases, the degree of $w$ in the hypergraph induced by $W$ is at least $k-1$.\\

\noindent Then, we add edges not contained in $W$.
If the degree of $v^*$ in $\Omega$ is at least $2$, there are at least  $2(|V_n|-k)$ distinct edges intersecting $V_n\setminus W_n$.
Otherwise, the parity property applied to $e(w^*)$, any $n$-transversal in $W$, and each vertex in $V_n\setminus W_n$ provides $|V_n|-k$ additional edges not intersecting $V_{n-1}\setminus W_{n-1}$. Therefore, an additional $|V_{n-1}|-k$  edges are needed to cover these vertices of  $V_{n-1}\setminus W_{n-1}$.\\

\noindent In total, we have at least $k(k-1)+|V_{n-1}|+|V_{n}|-2k$ 
edges.\\

\noindent 
{\em Case $(b)$:} $|V_{i^*}|=k-1$.  For each vertex $w\in V_{i^*}$, choose an edge $e(w)$ containing $w$. We can assume that there is a vertex $w^*\in V_{i^*}$ such that $e(w^*)$ contains a vertex $v^*$ in $V_n$ of minimal degree in $\Omega$. Choose $W_i\subseteq V_i$ such that $|W_i|=k-1$ for $i=1,\ldots,n-1$, $|W_n|=k$, and 
$$\bigcup_{w\in V_{i^*}} e(w) \subseteq W=\bigcup_{i=1}^{n} W_i.$$ 
 \noindent
Similarly, we show that the degree of any vertex in $V_{i^*}$ is at least $k-1$ in the hypergraph induced by $W$.  Pick $w\in V_{i^*}$ and consider $e(w)$. If $v\in e(w)$, take $k-2$ disjoint $i^*$-transversals in $W$ not containing $v'$ and not intersecting with $e(w)$.  Applying the parity property to $e(w)$, $u'$, and each of those $i^*$-transversals yields, in addition to $e(w)$, at least $k-2$ edges containing $w$. Otherwise, take $k-2$ disjoint $i^*$-transversals in $W$ not containing $v$ and not intersecting with $e(w)$, and apply the parity property to $e(w)$, $u$, and each of those $i^*$-transversals. Therefore, in both cases, the degree of $w$ in the hypergraph induced by $W$ is at least $k-1$.\\

\noindent Then, we add edges not contained in $W$.
If the degree of $v^*$ in $\Omega$ is at least $2$, there are at least  $2(|V_n|-k)$ distinct edges intersecting $V_n\setminus W_n$.
Otherwise, the parity property applied to $e(w^*)$, any $n$-transversal in $W$, and each vertex in $V_n\setminus W_n$ provides $|V_n|-k$ additional edges not intersecting $V_{n-1}\setminus W_{n-1}$. Therefore, $|V_{n-1}|-k+1$ additional edges are needed to cover these vertices of  $V_{n-1}\setminus W_{n-1}$.\\

\noindent In total, we have at least $(k-1)^2+|V_{n-1}|+|V_{n}|-2k$ 
edges. 
\qquad\end{proof}

\subsection{Proof of the main result}
\noindent 
Theorem~\ref{main} is obtained by setting $(k,z)=(m,0)$ in Proposition~\ref{prop:kzn}. This proposition is proven by induction on the cardinality of octahedral systems of the form illustrated in Figure~\ref{fig:induc}. Either the deletion of a vertex results in an octahedral system satisfying the condition of Proposition~\ref{prop:kzn} and we can apply induction, or we apply Lemma~\ref{lem:1in1} or Lemma~\ref{lem:2in1} to bound the number of edges of the system. Lemma~\ref{induc} is a key tool to determine if the deletion of a vertex results in an octahedral system satisfying the condition of Proposition~\ref{prop:kzn}. 

\begin{figure}
\begin{center}
\includegraphics[scale=0.6]{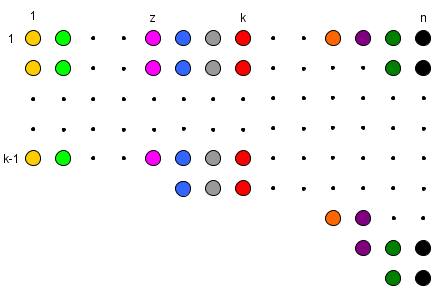}
\caption{The vertex set of the $(k-1,\ldots,k-1,k,\ldots,k,m_{k+1},\ldots,m_n)$-octahedral system $\Omega=(V_1,\ldots,V_n,E)$ used for the proof of Proposition~\ref{prop:kzn}}
\label{fig:induc}
\end{center}
\end{figure}

\begin{proposition}\label{prop:kzn}
A $(\overbrace{k-1,\ldots,k-1}^{z\mbox{ \scriptsize{\rm{times}}}},\overbrace{k,\ldots,k}^{k-z\mbox{ \scriptsize{\rm{times}}}},m_{k+1},\ldots,m_n)$-octahedral system $\Omega=(V_1,\ldots,V_n,E)$ without isolated vertex, with $2\leq k\leq m_{k+1}\leq\ldots\leq m_n$ and $0\leq z<k\leq n$, has at least 
$$\begin{array}{rl}
\vspace{0.1cm}\frac{1}{2}k^2+\frac{1}{2}k-8+|V_{n-1}|+|V_{n}|-z & \mbox{ edges if $k\leq n-2$},\\
\vspace{0.1cm}\frac{1}{2}n^2+\frac{1}{2}n-10+|V_n|-z & \mbox{ edges if $k=n-1$},\\
\frac{1}{2}n^2+\frac{5}{2}n-11-z & \mbox{ edges if $k=n$}.
\end{array}$$
\end{proposition}
\begin{proof}
The proof works by induction on $\sum_{i=1}^n|V_i|$.
The base case is $\sum_{i=1}^n|V_i|=2n$, which implies $z=0$ and $k=|V_{n-1}|=|V_{n}|=2$. The three inequalities trivially hold in this case.\\

\noindent 
Suppose that $\sum_{i=1}^n|V_i|>2n$. If $k=2$, Proposition~\ref{prop:nonisolated} proves the inequality. We can thus assume that $k\geq 3$.  We choose a pair $(k,z)$ compatible with $\Omega$ -- note that $(k,z)$ is not necessarily unique -- and we consider the two possible cases for the associated $D(\Omega)$.\\

\noindent 
If there are at least two vertices $v$ and $v'$ of $V_n$ having an outneighbour in the same $V_{i^*}$, with $i^*<k$, we can apply Lemma~\ref{lem:2in1}. If $k\leq n-2$, the inequality follows by a straightforward computation; if $k=n-1$, we use the fact that $|V_{n-1}|=n-1$; and if $k=n$, we use the fact that $|V_{n-1}|\geq n-1$ and $|V_n|=n$.\\

\noindent
Otherwise, for each $i<k$, there is at most one vertex of $V_n$ having an outneighbour in $V_i$.  Since $k-1<|V_n|$, there is a vertex $x$ of $V_n$ having no outneighbour in $\bigcup_{i=1}^{k-1}V_i$. Start from $x$ in $D(\Omega)$ till reaching a set $X$ inducing a complete subgraph of $D(\Omega)$ without outneighbour. Since $D(\Omega)$ is transitive, we have $X\subseteq\bigcup_{i=k}^n V_i$. If $|X|\geq 2$, we apply Lemma~\ref{lem:1in1}. Thus, we can assume that $|X|=1$.\\

\noindent
The hypergraph $\Omega'$ induced by $\Omega\setminus X$ is an octahedral system without isolated vertex since $X$ is a single vertex without outneighbour in $D(\Omega)$, see Lemma~\ref{induc}. Recall that the vertex in $X$ belongs to $\bigcup_{i=k}^nV_{i}$. Let $(k',z')$ be possible parameters associated to $\Omega'$ determined hereafter. Let $i_0$ be such that $X\subseteq V_{i_0}$. The induction argument is applied to the different values of $|V_{i_0}|$.\\

\noindent If $|V_{i_0}|\geq k+1$, we have $(k',z')=(k,z)$ and we can apply the induction hypothesis with 
$|V_{n-1}|+|V_{n}|$ decreasing by at most one (in case $i_0=n-1$ or $n$) which is compensated by the edge containing $X$.\\

\noindent If $|V_{i_0}|= k$, $z<k-1$, and $k<n$, we have $(k',z')=(k,z+1)$ and we can apply the induction hypothesis with same
$|V_{n-1}|$ and $|V_{n}|$ since $z< n-2$, while $z'$ replacing $z$ takes away 1 which is compensated by the edge containing $X$.\\

\noindent If $|V_{i_0}|= k$, $z=k-1$, and $k<n-1$,
we have $(k',z')=(k-1,0)$ and we can apply the induction hypothesis with same
$|V_{n-1}|+|V_{n}|$ since $z< n-2$. We get therefore $\frac{1}{2}(k-1)^2+\frac{1}{2}(k-1)-8+|V_{n-1}|+|V_n|$ edges in $\Omega'$, plus at least one containing $X$. In total, we have $\frac{1}{2}k^2+\frac{1}{2}k-8+|V_{n-1}|+|V_n|-k+1$ edges in $\Omega$, as required.\\

\noindent If $|V_{i_0}|= k$, $z=k-1$, and $k=n-1$,
we have $(k',z')=(n-2,0)$ and we can apply the induction hypothesis with
$|V_{n-1}|+|V_{n}|$ decreasing by at most one. We get therefore $\frac{1}{2}(n-2)^2+\frac{1}{2}(n-2)-8+|V_{n-1}|+|V_n|-1$ edges in $\Omega'$, plus at least one containing $X$. Since $|V_{n-1}|=n-1$, we have in total $\frac 1 2 n^2+\frac 1 2 n-10+|V_n|-(n-2)$ edges in $\Omega$, as required.\\

\noindent If $|V_{i_0}|= k$, $z=k-1$, and $k=n$,
we have then necessarily $i_0=n$ and $(k',z')=(n-1,0)$. We can apply the induction hypothesis and get therefore $\frac 1 2 n^2+\frac 1 2 n-10+(n-1)$ edges in $\Omega'$, plus at least one containing $X$. In total, we have $\frac 1 2 n^2+\frac 5 2 n-11-(n-1)$ edges in $\Omega$, as required.\\

\noindent If $|V_{i_0}|= k$, $z\leq k-2$, and $k=n$,
we can assume $i_0=n-1$ and $(k',z')=(n,z+1)$. We can apply the induction hypothesis and get therefore $\frac 1 2 n^2+\frac 5 2 n-11-z-1$ edges in $\Omega'$, plus at least one containing $X$. In total, we have $\frac 1 2 n^2+\frac 5 2 n-11-z$ edges in $\Omega$, as required.
\qquad\end{proof}

\begin{remark}{\rm
A similar analysis, with $|V_i|=n$ for all $i$ as a base case, shows that an octahedral system without isolated vertex and with $|V_1|=|V_2|=\ldots=|V_n|=m$ has at least  $nm-\frac{1}{2}n^2+\frac{5}{2}n-11$ edges for $4\leq n\leq m$.}
\end{remark}

\section{Small instances and $\mu(4)=17$}\label{mu4}
\noindent
This section focuses on octahedral systems with $m_i$'s and $n$ at most $5$. 

\begin{proposition}\label{prop:3333}
$\nu(3,3,3,3)=6$.
\end{proposition}
\begin{proof}
We first prove that $\nu(2,3,3,3)=5$. Let $\Omega=(V_1,V_2,V_3,V_4,E)$ be a $(2,3,3,3)$-octahedral system. In $D(\Omega)$ there is at most one vertex of $V_4$ having an outneighbour in $V_1$, otherwise one vertex of $V_4$ would be isolated.
Thus, there is a subset $X\subseteq V_2\cup V_3\cup V_4$ inducing in $D(\Omega)$
a complete subgraph without outneighbour.
If $|X|\geq 2$,  applying Lemma~\ref{tech} with $(k,z)=(3,1)$ gives at least $5$ edges in that case.
If  $|X|= 1$, deleting $X$ yields a $(2,2,3,3)$-octahedral system without isolated vertex since $X$ has no outneighbour in $D(\Omega)$.
As $\nu(2,2,3,3)=4$ by Proposition~\ref{prop:222m}, we have at least $4+1=5$ edges.
Thus, the equality holds since $\nu(2,3,3,3)\leq 5$ by Proposition~\ref{prop:upperbound_gen}.\\

\noindent
We then prove that $\nu(3,3,3,3)=6$. Let $\Omega=(V_1,V_2,V_3,V_4,E)$ be a $(3,3,3,3)$-octahedral system. There is a subset $X$ inducing in $D(\Omega)$ a complete subgraph without outneighbour.
If $|X|\geq 2$,  apply Lemma~\ref{tech} with $(k,z)=(3,0)$ gives at least $6$ edges in that case.
If  $|X|= 1$, deleting $X$ yields a $(2,3,3,3)$-octahedral system without isolated vertex since $X$ has no outneighbour in $D(\Omega)$.
As $\nu(2,3,3,3)=5$, we have at least $5+1=6$ edges.
Thus, the equality holds since $\nu(3,3,3,3)\leq 6$ by Proposition~\ref{prop:upperbound_gen}.
\qquad\end{proof}

The main result this section, namely  $\nu(5,5,5,5,5)=17$, is proven via a series of claims dealing with octahedral systems of increasing sizes. We first determine the values of $\nu(2,2,3,3,3)$, $\nu(2,3,3,3,3)$, and $\nu(3,3,3,3,3)$ in Claims~\ref{claim:22333},~\ref{claim:23333}, and~\ref{claim:33333}.
\begin{claim}\label{claim:22333}
$\nu(2,2,3,3,3)=5$.
\end{claim}
\begin{proof} 
For $i=1$ and $2$, there is at most one vertex of $V_5$ having an outneighbour in $V_i$ as otherwise one vertex of $V_5$ would be isolated. Since $|V_5|= 3$, there is a vertex of $V_5$ having no outneighbour in $V_1\cup V_2$. Thus, there is a subset $X\subseteq V_3\cup V_4\cup V_5$ of cardinality $1$, $2$, or $3$ inducing a complete subgraph in $D(\Omega)$ without outneighbour.
If $|X|\geq 2$,  applying Lemma~\ref{tech} with $(k,z)=(3,2)$ gives at least $5$ edges. If  $|X|= 1$, deleting $X$ yields a $(2,2,2,3,3)$-octahedral system without isolated vertex since $X$ has no outneighbour in $D(\Omega)$. As $\nu(2,2,2,3,3)=4$ by Proposition~\ref{prop:222m}, we have at least $4+1=5$ edges. Thus, the equality holds since $\nu(2,2,3,3,3)\leq 5$ by Proposition~\ref{prop:upperbound_gen}.
\qquad\end{proof}
\begin{claim}\label{claim:23333}
$\nu(2,3,3,3,3)=6$.
\end{claim}
\begin{proof} 
We consider  the two possible cases for the associated $D(\Omega)$.\\\\
{\em Case $(a)$:} there are at least two vertices $v$ and $v'$ of $V_5$ having outneighbours in the same $V_{i^*}$ in $D(\Omega)$ with $i^*=1$ or $2$. 
Note that actually $i^*=2$ since otherwise $V_5\setminus\{v,v'\}$ would be isolated.
Applying Lemma~\ref{lem:2in1} with $(k,z)=(3,1)$ gives at least $3\times 2+|V_4|+|V_5|-6=6$ edges.\\\\
{\em Case $(b)$:} there is at most one vertex of $V_5$ having an outneighbour in $V_i$ for $i=1$ and $2$ in $D(\Omega)$. 
Since $|V_5|=3$, there is a vertex of $V_5$ having no outneighbour in $V_1\cup V_2$. Thus, there is a subset $X\subseteq V_3\cup V_4\cup V_5$ inducing in $D(\Omega)$ a complete subgraph without outneighbour. 
If $|X|\geq 2$,  applying Lemma~\ref{tech} with $(k,z)=(3,1)$ and $j=2$ gives at least $6$ edges. 
If  $|X|= 1$, deleting $X$ yields a $(2,2,3,3,3)$-octahedral system without isolated vertex since $X$ has no outneighbour in $D(\Omega)$.
As $\nu(2,2,3,3,3)=5$ by Claim~\ref{claim:22333}, we have at least $5+1=6$ edges.\\\\
Thus, the equality holds since $\nu(2,3,3,3,3)\leq 6$ by Proposition~\ref{prop:upperbound_gen}.
\qquad\end{proof}
\begin{claim}\label{claim:33333}
$\nu(3,3,3,3,3)=7$.
\end{claim}
\begin{proof}
There is a subset $X$ inducing a complete subgraph in $D(\Omega)$ without outneighbour. Choose such an $X$ of maximal cardinality. Without loss of generality, we assume that the indices $i$ such that $|X\cap V_i|\neq 0$ are $n-|X|+1,n-|X|+2,\ldots,n$. Consider the different values for $|X|$.
\begin{itemize}
\item[$\bullet$]
If $|X|=1$, deleting $X$ yields a $(2,3,3,3,3)$-octahedral system without isolated vertex since $X$ has no outneighbour in $D(\Omega)$. As $\nu(2,3,3,3,3)=6$ by Claim~\ref{claim:23333}, we have at least $6+1=7$ edges.
\item[$\bullet$] If $|X|=2$ and $\deg_{\Omega}(X)\geq 2$, deleting $X$ yields a $(2,2,3,3,3)$-octahedral system without isolated vertex. As $\nu(2,2,3,3,3)=5$ by Claim~\ref{claim:22333}, we have at least $5+2=7$ edges.
\item[$\bullet$] If $|X|=2$ and $\deg_{\Omega}(X)=1$,  denote $e(X)$ the unique edge containing $X$. 
For  $i=1,2$, and $3$, pick a vertex $w_i$ in $V_i\setminus e(X)$. Applying the parity property to $e(X)$, $w_1$, $w_2$, $w_3$, and any $u_4\in V_4\setminus e(X)$, $u_5\in V_5\setminus e(X)$
 yields at least $5$ edges in $e(X)\cup\{w_1,w_2,w_3\}\cup V_4\cup V_5$. At least $2$ additional edges are needed to cover the $3$ remaining vertices of $V_1$, $V_2$, and $V_3$ since a unique edge containing them would contradict the maximality of $X$. Thus, we have at least $7$ edges.
\item[$\bullet$] If $|X|=3$ and $\deg_{\Omega}(X)\geq 3$, deleting $X$ yields a $(2,2,2,3,3)$-octahedral system without isolated vertex. As $\nu(2,2,2,3,3)=4$ by Proposition~\ref{prop:222m}, we have at least $4+3=7$ edges. 
\item[$\bullet$] If $|X|=3$ and $\deg_{\Omega}(X)\leq 2$,  let $e(X)$ be an edge containing $X$. Pick  $w_1\in V_1\setminus N_{\Omega}(X)$ and $w_2\in V_2\setminus N_{\Omega}(X)$ where $N_{\Omega}(X)$ denotes the vertices not in $X$ contained in the edges intersecting $X$. Applying the parity property to $e(X)$, $w_1$, $w_2$, and any $u_i\in V_i\setminus e(X)$ for $i=3,4$, and $5$  yields at  least $9$ edges in $e(X)\cup\{w_1,w_2\}\cup V_3\cup V_4\cup V_5$.
\item[$\bullet$] If $|X|=4$ and $\deg_{\Omega}(X)\geq 3$, take any vertex $v$ in $V_2\setminus X$. 
Applying the parity property to an edge $e(v)$ containing $v$, $V_2\cap X$, and any $2$-transversal disjoint from $e(v)$ and $X$ shows that $v$ is of degree at least $2$. Since there are $2$ vertices in $V_2\setminus X$, we get, with $3$ edges containing $X$, at least $7$ edges.
\item[$\bullet$] If $|X|=4$ and $\deg_{\Omega}(X)\leq 2$,  let $e(X)$ be an edge containing $X$. Pick $w_1\in V_1\setminus N_{\Omega}(X)$. Applying the parity property to $e(X)$, $w_1$, and any $u_i\in V_i\setminus e(X)$ for $i=2,3,4$, and $5$ yields at least $17$ edges in $e(X)\cup\{w_1\}\cup V_2\cup V_3\cup  V_4\cup V_5$. 
\item[$\bullet$] If $|X|=5$, the parity property applied to the edge $e(X)$ containing $X$, and any $u_i\in V_i\setminus e(X)$ for $i=1,2,3,4$, and $5$ yields  at least $33$ edges.
\end{itemize}
Thus, the equality holds since $\nu(3,3,3,3,3)\leq 7$ by Proposition~\ref{prop:upperbound_gen}.
\qquad\end{proof}

To complete the proof of $\nu(5,5,5,5,5)=17$, we sequentially show $\nu(3,3,3,3,4)\geq 7$, $\nu(4,4,4,4,4)=12$, and finally $\nu(5,5,5,5,5)=17$. A key step consists in proving $\nu(4,4,4,4,4)\geq 11$ 
by induction using $\nu(3,3,3,3,4)\geq 7$ as a base case. We obtain then $\nu(4,4,4,4,4)=12$ by 
Propositions~\ref{parity} and~\ref{prop:upperbound_gen}. The equality $\nu(5,5,5,5,5)=17$ is obtained by induction using $\nu(4,4,4,4,4)=12$ as a base case.
\begin{claim}\label{claim:23334}
$\nu(3,3,3,3,4)\geq 7$.
\end{claim}
\begin{proof}
We first prove 
$\nu(2,3,3,3,4)\geq 6$ which in turn leads to 
$\nu(3,3,3,3,4)\geq 7$. The proof of these two inequalities are quite similar with the main difference being that, while the first inequality relies partially on Proposition~\ref{prop:nonisolated},  the second inequality relies on the first one.\\

\noindent 
Let $\Omega=(V_1,\ldots,V_5,E)$ be a $(2,3,3,3,4)$- or a $(3,3,3,3,4)$-octahedral system. We consider the three possible cases for the associated $D(\Omega)$.\\

\noindent 
{\em Case $(a)$:} there is a vertex of $V_5$ having no outneighbour. Deleting this vertex yields a $(2,3,3,3,3)$- or a $(3,3,3,3,3)$-octahedral system  without isolated vertex. In both cases, we have at least $7$ edges since $\nu(2,3,\ldots,3)=6$ by Claim~\ref{claim:23333}, and $\nu(3,3,3,3,3)=7$ by Claim~\ref{claim:33333}.\\

\noindent
{\em Case $(b)$:} each vertex of $V_5$ has an outneighbour and there are at least two vertices $v$ and $v'$ of $V_5$ having outneighbours in the same $V_{i^*}$ in $D(\Omega)$ with $i^*=1,2$, or $3$. 
Note that $|V_{i^*}|=3$ since otherwise $V_5\setminus\{v,v'\}$ would be isolated. Applying Lemma~\ref{lem:2in1} with either $(k,z)=(3,1)$ or $(k,z)=(3,0)$ gives at least $3\times 2+|V_4|+|V_5|-6=7$ edges.\\

\noindent
{\em Case $(c)$:}  each vertex of $V_5$ has an outneighbour and there is at most one vertex of $V_5$ having an outneighbour in $V_i$ for $i=1,2$, and $3$. 
Since $|V_5|=4$, there is a subset $X\subseteq V_4\cup V_5$ inducing in $D(\Omega)$ a complete subgraph of cardinality $1$ or $2$ without outneighbour.
\begin{itemize}
\item[$\bullet$]If $|X|=1$, we have $X\subseteq V_4$ since each vertex of $V_5$ has an outneighbour. Deleting $X$ yields a $(2,2,3,3,4)$- or a $(2,3,3,3,4)$-octahedral system without isolated vertex. 
We obtain $\nu(2,3,3,3,4)\geq 6$ since $\nu(2,2,3,3,4)\geq 5$ by Proposition~\ref{prop:nonisolated}, and then $\nu(3,3,3,3,4)\geq 7$ since $\nu(2,3,3,3,4)\geq 6$.
\item[$\bullet$] If $|X|=2$, deleting $X$ yields a $(2,2,3,3,3)$- or a $(2,3,3,3,3)$-octahedral system without isolated vertex. 
Since one additional edge is needed to cover $X$, 
we obtain $\nu(2,3,3,3,4)\geq 6$ since $\nu(2,2,3,3,3)=5$ by Claim~\ref{claim:22333}, and $\nu(3,3,3,3,4)\geq 7$ since $\nu(2,3,3,3,3)=6$ by Claim~\ref{claim:23333}.
\end{itemize}
\qquad\end{proof}

\begin{claim}\label{claim:34}
$\nu(\underbrace{3,\ldots,3}_{z\mbox{ \scriptsize{\rm{times}}}},\underbrace{4,\ldots,4}_{5-z\mbox{ \scriptsize{\rm{times}}}})\geq 11-z$ for $z=1,2,3$.
\end{claim}
\begin{proof}
The proof works by induction on $z$ using the inequality $\nu(3,3,3,3,4)\geq 7$ which holds by Claim~\ref{claim:23334}.
We consider  the two possible cases for the associated $D(\Omega)$.\\

\noindent 
{\em Case $(a)$:} there are at least two vertices $v$ and $v'$ of $V_5$ having outneighbours in the same $V_{i^*}$ with $i^*\leq z$. Let $u$ and $u'$ be the two vertices in $V_{i^*}$ with $(v,u)$ and $(v',u')$ forming arcs in $D(\Omega)$. For each vertex $w\in V_{i^*}$, choose an edge $e(w)$ containing $w$. Choose $W_i\subseteq V_i$ such that $|W_i|=3$ for $i=1,\ldots,4$, $|W_5|=4$, and $$\bigcup_{w\in V_{i^*}} e(w) \subseteq W=\bigcup_{i=1}^{5} W_i.$$ 
 \noindent
Pick $w\in V_{i^*}$ and consider $e(w)$. If $v\in e(w)$, take $2$ disjoint $i^*$-transversals in $W$ not containing $v'$ and not intersecting with $e(w)$.  Applying the parity property to $e(w),u'$, and each of those $i^*$-transversals yields, in addition to $e(w)$, at least $2$ edges containing $w$. Otherwise, take $2$ disjoint $i^*$-transversals in $W$ not containing $v$ and not intersecting with $e(w)$, and apply the parity property to $e(w),u$, and each of those $i^*$-transversals. In both cases, the degree of $w$ in the hypergraph induced by $W$ is at least $3$. Then, we add edges not contained in $W$. Since $V_4\setminus W\neq\emptyset$, there is at least one additional edge. In total, we have at least $10\geq 11-z$ edges.\\

\noindent
{\em Case $(b)$:} there is at most one vertex of $V_5$ having an outneighbour in $V_i$ for $i\leq z$. Since $|V_5|=4$, there is at least one vertex of $V_5$ having no outneighbour in $\bigcup_{i=1}^zV_i$.  Thus, there is a subset $X\subseteq\bigcup_{i=z+1}^5V_i$ inducing in $D(\Omega)$ a complete subgraph without outneighbour. If $|X|=1$, deleting $X$ yields a $(\underbrace{3,\ldots,3}_{z+1\mbox{ \scriptsize{times}}},\underbrace{4,\ldots,4}_{4-z\mbox{ \scriptsize{times}}})$-octahedral system without isolated vertex. As $\nu(\underbrace{3,\ldots,3}_{z+1\mbox{ \scriptsize{times}}},\underbrace{4,\ldots,4}_{4-z\mbox{ \scriptsize{times}}})\geq 11-(z+1)$ we obtain $11-z$ edges. 
If $|X|\geq 2$, we have at least $9+2=11$ edges by Lemma~\ref{tech} with $(k,z)=(4,z)$.
\qquad\end{proof}

\begin{claim}\label{claim:44444}
$\nu(4,4,4,4,4)=12$.
\end{claim}
\begin{proof} 
There is a subset $X$ inducing a complete subgraph in $D(\Omega)$ without outneighbour. If $|X|=1$, deleting $X$ yields a $(3,4,\ldots,4)$-octahedral system without isolated vertex.  As $\nu(3,4,\ldots,4)\geq 10$, we obtain $11$ edges. If $|X|\geq 2$, we have at least $11$ edges by Lemma~\ref{tech} with $(k,z)=(4,0)$. Thus,  $\nu(4,4,4,4,4)\geq 12$ by Proposition~\ref{parity}, and then $\nu(4,4,4,4,4)=12$ by Proposition~\ref{prop:upperbound_gen}.
\qquad\end{proof}
\begin{claim}\label{claim:xx}
$\nu(\underbrace{4,\ldots ,4}_{z\mbox{ \scriptsize{\rm times}}},\underbrace{5,\ldots ,5}_{5-z\mbox{ \scriptsize{\rm times}}} )=17-z$ for $z=1,2,3,4$.
\end{claim}
\begin{proof}
The proof works by induction on $z$ using the inequality $\nu(4,4,4,4,4)\geq 12$ which holds by Claim~\ref{claim:44444}. 
We consider the two possible cases for the associated $D(\Omega)$.\\

\noindent
{\em Case $(a)$:} there are at least two vertices $v$ and $v'$ of $V_5$ having outneighbours in the same $V_{i^*}$ with $i^*\leq z$.
We can apply Lemma~\ref{lem:2in1} with $(k,z)=(5,z)$, we have at least $4\times 4 + |V_4| +|V_5|-10\geq 17-z$ edges.\\

\noindent
{\em Case $(b)$:} there is at most one vertex of $V_5$ having an outneighbour in $V_i$ for $1\leq i\leq z$. Since $|V_5|=5$, there is a vertex of $V_5$ having no outneighbour in $\bigcup_{i=1}^zV_i$. Thus, there is a subset $X\subseteq \bigcup_{i=z+1}^5V_i$ inducing in $D(\Omega)$ a complete subgraph without outneighbour.
If $|X|=1$, deleting $X$ yields a $(\underbrace{4,\ldots ,4}_{z+1 \mbox{ \scriptsize times }},\underbrace{5,\ldots ,5}_{4-z \mbox{ \scriptsize times }} )$-octahedral system. As $\nu(\underbrace{4,\ldots ,4}_{z+1 \mbox{ \scriptsize times }},\underbrace{5,\ldots ,5}_{4-z \mbox{ \scriptsize times }} )\geq 16-z$, we obtain at least $17-z$ edges.
If $|X|\geq 2$, we have at least $18$ edges by Lemma~\ref{tech}.\\

\noindent
Thus, the equality holds since $\nu(\underbrace{4,\ldots ,4}_{z \mbox{ \scriptsize times }},\underbrace{5,\ldots ,5}_{5-z \mbox{ \scriptsize times }} )\leq 17-z$ by Proposition~\ref{prop:upperbound_gen}.
\qquad\end{proof}

\begin{claim}\label{claim:55555}
$\nu(5,5,5,5,5)=17$.
\end{claim}
\begin{proof}
There is a subset $X$ inducing a complete subgraph in $D(\Omega)$ without outneighbour. If $|X|=1$, deleting $X$ yields a $(4,5,5,5,5)$-octahedral system without isolated vertex. As $\nu(4,5,5,5,5)\geq 16$, we have at least $17$ edges. If $|X|\geq 2$, we can apply Lemma~\ref{tech}, and we have  at least $18$ edges.\\

\noindent
Thus, the equality holds since $\nu(5,5,5,5,5)\leq 17$ by Proposition~\ref{prop:upperbound_gen}.
\qquad\end{proof}

\noindent
As $\nu(5,5,5,5,5)=\nu(4)$, Claim~\ref{claim:55555} and the relation $\mu(4)\geq\nu(4)$  directly imply that the conjectured equality $\mu(d)=d^2+1$ holds for $d=4$.
\begin{proposition}\label{prop:mu4}
$\mu(4)=17$.
\end{proposition}

\bibliographystyle{siam}
\bibliography{refs}

\end{document}